\documentclass[final]{siamltex}

\pdfoutput=1

\usepackage{epsfig}
\usepackage[english]{babel}
\usepackage{amssymb,amsmath,amsfonts}
\usepackage{mathrsfs}
\usepackage{mdwlist}
\usepackage{graphicx,enumerate,url,hyperref,enumitem,multirow,multicol}
\usepackage[usenames,dvipsnames,svgnames]{xcolor}
\hypersetup{pdfauthor={Feng Xue},
    pdftitle={The nonexpansive operator with degenerate metric},
    colorlinks=true,
    linkcolor=blue,
    citecolor=red,
    urlcolor=blue,
}

\usepackage{algorithm}
\usepackage{algcompatible}
\algnewcommand\algorithmicreturn{\textbf{return}}
\algnewcommand\RETURN{\State \algorithmicreturn}%
\usepackage[outercaption]{sidecap}

\usepackage{tikz}
\usepackage{bm}

\usepackage{pgfplots}
\pgfplotsset{
  tick label style={font=\footnotesize},
  label style={font=\footnotesize},
  legend style={font=\footnotesize}
}

\usepackage{booktabs}
\usepackage{caption,subcaption}

\usepackage{todonotes}

\newcommand{\bbstar}{x^\star}

\newcommand{\bb}{x}

\newcommand{\ran}{\mathsf{ran}}

\newcommand{\Fix}{\mathsf{Fix}}

\newcommand{\R}{\mathbb{R}}
\newcommand{\N}{\mathbb{N}}
\newcommand{\cI}{\mathcal{I}}
\newcommand{\cR}{\mathcal{R}}
\newcommand{\cS}{\mathcal{S}}

\newcommand{\cO}{\mathcal{O}}
\newcommand{\cT}{\mathcal{T}}
\newcommand{\cA}{\mathcal{A}}

\newcommand{\cF}{\mathcal{F}}

\newcommand{\cP}{\mathcal{P}}

\newcommand{\cH}{\mathcal{H}}

\newcommand{\cM}{\mathcal{M}}

\newcommand{\cQ}{\mathcal{Q}}

\newcommand{\cK}{\mathcal{K}}

\newcommand{\gra}{\mathrm{gra}}
\newcommand{\zer}{\mathsf{zer}}

\newcommand{\kersf}{\text{\sf ker}}

\newcommand{\be}{\begin{equation}}
\newcommand{\ee}{\end{equation}}

\DeclareMathOperator{\prox}{prox}       
\newtheorem{assumption}{Assumption}

\newtheorem{remark}{Remark}
\newtheorem{fact}{Fact}
\usepackage[normalem]{ulem} 

\newif\ifcompilePGFfigs
\compilePGFfigstrue

\graphicspath{{figs/}}

\title{On the nonexpansive operators based on arbitrary metric: A degenerate analysis}

\author{Feng Xue\thanks{National Key Laboratory, Beijing, China (\tt{fxue@link.cuhk.edu.hk}).}  }
\date{\today}

\begin{document}

\maketitle

\begin{abstract}
We in this paper  study the nonexpansive operators equipped with arbitrary  metric and investigate the connections between firm nonexpansiveness, cocoerciveness and averagedness. The convergence of the associated fixed-point iterations is discussed with particular focus on the case of degenerate metric, since the degeneracy is often encountered when reformulating many existing first-order operator splitting algorithms as a metric resolvent. This work paves a way for analyzing the generalized proximal point algorithm with  a non-trivial relaxation step and degenerate metric. 
\end{abstract}



\section{Introduction}
\subsection{Nonexpansive operators}
The nonexpansive mappings were extensively studied in some early works, e.g. \cite{baillon,borwein,kirk}, and the generalizations have recently been discussed in \cite{res_17,
res_14,res_16}. Refer to \cite{plc_book,hhb_corres} for the comprehensive treatments.

The notion of nonexpansiveness arises primarily in connection with the study of fixed-point theory, and underlies the convergence analysis of various fixed-point iterations.  Nowadays, there has been a revived interest in the design and analysis of the first-order operator splitting methods \cite{teboulle_2018}, of which many algorithms can be interpreted by the nonexpansive mappings, e.g. proximal forward-backward splitting algorithms \cite{attouch_2018,plc,prox_acha},  Douglas-Rachford splitting  \cite{boyd_control},  primal-dual splitting methods \cite{vu_2013,bot_2014}.  An analysis of the operator splitting algorithms from the perspective of nonexpansive mappings is given by \cite{ljw_mapr}, which reinterprets a variety of algorithms by a simple Krasnosel'ski\u{\i}-Mann iteration built from a nonexpansive operator. More recently, the work of \cite{plc_fixed} gives a systematic overview of operator splitting algorithms based on the  fixed-point theory. This demonstrates that  the nonexpansive mapping still plays a central role and constantly gains the popularity and attention in the related areas.

Recently, the nonexpansive mappings have been extended to arbitrary self-adjoint and positive definite (PD) metric in various specific forms, e.g. generalized proximity operator \cite{prox_acha}, Bregman-based proximal operator \cite{entropic_tsp,bregman_vietnam}, generalized resolvent \cite{res_3,hhb_resolvent}, which are fundamental for analyzing proximal mapping \cite{chaokan}, Bregman-based proximal schemes \cite{teboulle_2018}, variable metric Fej\'{e}r sequence \cite{plc_vu}, especially under the context of operator splitting algorithms.

\subsection{Motivations and contributions}
\label{sec_mot}
All of the existing works assume the metric to be self-adjoint and PD, e.g. \cite{arias_infimal,hhb_resolvent,warp,fxue_gopt}. However, in some scenarios, especially when the operator splitting schemes are reinterpreted by the generalized proximal point algorithm (with a non-trivial relaxation step\footnote{The `{\it non-trivial}' means that the relaxation operator is not  multiple of identity operator, see \cite[Eq.(2.8)]{hbs_jmiv_2017}, \cite[Eq.(3.5)]{hbs_prs} and \cite[Eq.(5.2)]{hbs_yxm_2015} for example.}), one has to establish the (firmly) nonexpansive results based on a non-self-adjoint operator (see \cite[Lemma 3.2]{hbs_siam_2012} and \cite[Theorem 3.1, Lemma 5.4]{hbs_jmiv_2017} for example). The intermediate results are essential for proving the convergence. In Sect. \ref{sec_operator}, we study the nonexpansive properties (e.g. firm nonexpansiveness, cocoerciveness and averagedness)  in the context of  arbitrary  metric, and present the weakest possible conditions under which those properties hold. Indeed, many of the properties are shown to be valid without the self-adjoint and PD condition. 

It is easy and straightforward   to extend many convergence properties presented in \cite{plc_book,ppa_guler} to the positive definite metric, by simply replacing ordinary norm $\|\cdot\|$ by a metric-based semi-norm \cite{bredies_2017,fxue_gopt}. However, it would be non-trivial to extend the classic results to the case of positive semi-definite (PSD) metric, since the distance in the PSD metric cannot measure the closeness between two points in whole space due to the non-trivial null space of the metric. We refer to the positive semi-definiteness as `{\it degeneracy}', which  implies that the degenerate metric-based nonexpansiveness can infer the convergence in a subspace only, but not in the whole space if without additional assumptions.  With particular focus on degenerate case, we in Sect. \ref{sec_fixed} analyze the properties of the metric distance for the fixed-point iterations, and further prove the convergence in the whole space under a certain mild conditions.


All of the results presented in Sect. \ref{sec_operator} and \ref{sec_fixed} have important applications, of which a prominent example is the metric resolvent. Many operator splitting algorithms can be reinterpreted as the (degenerate) metric resolvents, and thus, the convergence properties can be easily obtained from the above results. This will be discussed in Sect. \ref{sec_appl}.

\subsection{Notations} \label{sec_notation}
We use standard notations and concepts from convex analysis and variational analysis, which, unless otherwise specified, can all be found in the classical and recent monographs \cite{rtr_book,rtr_book_2,plc_book,beck_book}.

A few more words about our  notations are in order. Let $\cH$ be a real Hilbert space, $D$ be a nonempty subset of $\cH$. $\cP_C$ denotes a projection onto a closed subset $C \subset\cH$.  The classes of PSD/PD linear operators are denoted by $\cM^+/\cM^{++}$, respectively.  The classes of self-adjoint, self-adjoint and PSD/PD linear operators are denoted by $\cM_\cS$, $\cM_\cS^+/\cM_\cS^{++}$, respectively.  For our specific use, the $\cQ$-norm for arbitrary metric $\cQ$ is defined as: $\|\cdot\|_\cQ^2 :=  \langle \cQ\cdot | \cdot \rangle$. Here,  $\cQ$ is {\it not} assumed to be self-adjoint and PSD, and hence, $\|\cdot\|_\cQ$ is not always well-defined, which is used only when $\cQ \in \cM^+$.  The strong and weak convergences are denoted by $\rightarrow$ and $\rightharpoonup$, respectively.

\section{The nonexpansive properties based on arbitrary  metric}
\label{sec_operator}
Throughout this paper, we assume that the nonexpansive operator $\cT$ is single-valued, but not necessarily injective.
In fact, $\cT$ is generally non-injective in the case of degenerate metric (see Sect. \ref{sec_appl} for example).

\subsection{Definitions based on arbitrary  metric}
Inspired by the work of \cite{bredies_2017}, the following definitions extend the classical notions of  Lipschitz continuity \cite[Definition 1.47]{plc_book}, nonexpansiveness \cite[Definition 4.1]{plc_book}, cocoerciveness \cite[Definition 4.10]{plc_book} and averagedness \cite[Definition 4.33]{plc_book} 
 to arbitrary ({\it not necessarily self-adjoint and PSD}) metric $\cQ$.

\begin{definition} \label{def_nonexpansive}
{\rm 
Let $\cQ$ be arbitrary  metric, then, the operator $\cT: D\mapsto \cH$ is:

(i) {\it $\cQ$--partly nonexpansive}, if:
\[
 \langle \cQ (\bb_1 - \bb_2) | \cT\bb_1-\cT\bb_2 \rangle
 \ge  \big\| \cT \bb_1 - \cT \bb_2  \big\|_\cQ^2.
 \]

(ii) {\it $\cQ$--nonexpansive}, if:
\[
 \big\| \cT \bb_1 - \cT \bb_2  \big\|_\cQ^2 \le 
 \big\|  \bb_1 - \bb_2 \big\|_\cQ^2.
 \]

(iii) {\it $\cQ$-based $\xi$--Lipschitz continuous}, if:
\[
 \big\| \cT \bb_1 - \cT \bb_2  \big\|_\cQ^2 \le 
\xi^2  \big\|  \bb_1 - \bb_2 \big\|_\cQ^2.
 \]

(iv) {\it $\cQ$--firmly nonexpansive}, if: 
\[
\big\| \cT\bb_1-\cT\bb_2 \big\|_\cQ^2 +  
\big\|(\cI - \cT) \bb_1 - (\cI - \cT) \bb_2  \big\|_\cQ^2 \le 
 \big\| \bb_1 - \bb_2  \big\|_\cQ^2.
 \]

(v) {\it $\cQ$--based $\beta$--cocoercive}, if $\beta \cT$ is $\cQ$--partly nonexpansive:
\[
\langle \cQ (\bb_1 - \bb_2) | \cT\bb_1-\cT\bb_2 \rangle
\ge \beta \big\| \cT \bb_1 - \cT \bb_2  \big\|_\cQ^2.
 \]

(vi) {\it $\cQ$--based $\alpha$--averaged} with $\alpha \in \ ]0,1 [$, if there exists a $\cQ$--nonexpansive operator $\cK: D \mapsto \cH$, such that $\cT = (1-\alpha) \cI + \alpha \cK$.
}
\end{definition}

\begin{remark}
The notions of {\it partly nonexpansive} and {\it $\beta$--cocoercive} are based on arbitrary metric $\cQ$, not limited to self-adjoint and PSD case.  Definition \ref{def_nonexpansive}--(ii) and (iii) use $\|\cdot\|_\cQ^2$ rather than $\|\cdot\|_\cQ$, because $\|\cdot\|_\cQ^2$ is always well defined for arbitrary $\cQ$, even if $\cQ$ is not PSD, as mentioned in Sect. \ref{sec_notation}.
\end{remark}

\vskip.1cm
The {\it  $\cQ$--based $\alpha$--averaged} is further generalized  as follows.
\begin{definition} \label{def_lip} {\rm
An operator $\cT: D \mapsto \cH$ is said to be  {\it  $\cQ$--based $\xi$--Lipschitz $\alpha$--averaged} with $\xi \in \ ]0,+\infty[$ and $\alpha\in \ ]0,1[$, if there exists a $\cQ$--based $\xi$--Lipschitz continuous operator $\cK: D \mapsto \cH$, such that $\cT = (1-\alpha) \cI + \alpha \cK$. In particular, if $\xi \in \ ]1, +\infty[$, $\cT$ is {\it $\cQ$--weakly averaged}; if $\xi \in \  ]0,1]$,  $\cT$ is {\it $\cQ$--strongly averaged}.
}
\end{definition}

To lighten the notation, we denote a family of $\cQ$--based $\xi$--Lipschitz $\alpha$--averaged operators  by $\cF^\cQ_{\xi,\alpha}$. Obviously, Definition \ref{def_nonexpansive}--(vi) is a special case of Definition \ref{def_lip} with  $\xi=1$, and thus, can be denoted by  $\cT \in \cF^\cQ_{1,\alpha}$.

\subsection{Nonexpansiveness}
This section presents the nonexpansive properties in the context of arbitrary  metric $\cQ$.  First, Definition \ref{def_nonexpansive} is connected via the following results.
\begin{lemma} \label{l_part}
$\cT: D \mapsto \cH$ is $\cQ$--partly nonexpansive, 

{\rm (i)} if and only if $\cT$ is $\cQ$--based 1--cocoercive;

{\rm (ii)} if $\cT$ is $\cQ$--based $\beta$--cocoercive with $\cQ \in \cM^+$ and $\beta  \in [1, +\infty[  $;

{\rm (iii)} if and only if $\cT$ is $\cQ$--firmly nonexpansive with $\cQ \in \cM_\cS$.
\end{lemma}
\begin{proof}
Definition \ref{def_nonexpansive} and the conditions of $\cQ$.
\hfill 
\end{proof}

\begin{lemma} \label{l_non}
$\cT: D \mapsto \cH$ is $\cQ$--nonexpansive, 

{\rm (i)} if and only if $\cT$ is $\cQ$--based 1--Lipschitz continuous;

{\rm (ii)} if $\cT$ is $\cQ$--based $\xi$--Lipschitz continuous with $\cQ \in \cM^+$ and $\xi \in \ ]0, 1]$;

{\rm (iii)} if $\cT$ is $\cQ$--firmly nonexpansive with $\cQ \in \cM^+$.
\end{lemma}
\begin{proof}
Definition \ref{def_nonexpansive} and the conditions of $\cQ$.
\hfill 
\end{proof}

\cite[Corollary 2.15]{plc_book} is also valid for arbitrary ({\it not necessarily self-adjoint and PSD}) metric $\cQ$, as stated below.
\begin{lemma} \label{l_id_1}
The following identity holds for any $\kappa \in\R$ and arbitrary $\cQ$:
\[
\big\| \kappa \bb_1 +(1-\kappa) \bb_2 \big\|_\cQ^2
= \kappa  \big\| \bb_1  \big\|_\cQ^2
+ (1-\kappa) \big\| \bb_2 \big\|_\cQ^2
- \kappa(1-\kappa)  \big\| \bb_1 - \bb_2 \big\|_\cQ^2.
\]
\end{lemma}
\begin{proof}
Noting that $\langle (\cQ+\cQ^*) \bb_1 |  \bb_2 \rangle= \|\bb_1\|_\cQ^2 + \|\bb_2\|_\cQ^2 - \|\bb_1 - \bb_2\|_\cQ^2$, we have:
\begin{eqnarray}
&& \big\| \kappa \bb_1 +(1-\kappa) \bb_2 \big\|_\cQ^2
\nonumber \\
&=& \kappa^2  \big\| \bb_1  \big\|_\cQ^2
+ (1-\kappa)^2 \big\| \bb_2 \big\|_\cQ^2
+ \kappa(1-\kappa) \langle (\cQ+\cQ^*)\bb_1 |  \bb_2 \rangle
\nonumber \\
&=& \kappa^2  \big\| \bb_1  \big\|_\cQ^2
+ (1-\kappa)^2 \big\| \bb_2 \big\|_\cQ^2
+ \kappa(1-\kappa) \big(  \|\bb_1\|_\cQ^2 + \|\bb_2\|_\cQ^2 - \|\bb_1 - \bb_2\|_\cQ^2 \big)
\nonumber \\
&=& \kappa  \big\| \bb_1  \big\|_\cQ^2
+ (1-\kappa) \big\| \bb_2 \big\|_\cQ^2
- \kappa(1-\kappa)  \big\| \bb_1 - \bb_2 \big\|_\cQ^2, 
\nonumber 
\end{eqnarray}
which completes the proof. \hfill 
\end{proof}

Lemma \ref{l_eq} is an extended version of  \cite[Proposition 4.2]{plc_book} for the case of arbitrary metric $\cQ$, which shows the equivalence between {\it partly nonexpansive} and {\it firmly nonexpansive}, in case of self-adjoint  $\cQ$. 
\begin{lemma} \label{l_eq}
Let $\cT: D \mapsto \cH$, then,
the following are equivalent:

{\rm (i)} $\cT$ is $\cQ$--partly nonexpansive;

{\rm (ii)} $\cT$ is $\cQ$--firmly nonexpansive with $\cQ \in \cM_\cS$; 

{\rm (iii)} $\cI - \cT$ is $\cQ$--firmly nonexpansive with $\cQ \in \cM_\cS$; 

{\rm (iv)} $2\cT - \cI$ is $\cQ$--nonexpansive with $\cQ \in \cM_\cS$;

{\rm (v)} $\cI - \cT$ is $\cQ^*$--partly nonexpansive. 
\end{lemma}
\begin{proof}
(i)$\leftrightarrow$(ii)$\leftrightarrow$(iii)$\leftrightarrow$(iv):  Definition \ref{def_nonexpansive}, \cite[Proposition 4.2]{plc_book} and Lemma \ref{l_id_1}.

(i)$\leftrightarrow$(v): By Definition \ref{def_nonexpansive}--(i),  we have:
\[
\big \langle \cQ (\cI - \cT) \bb_1 - \cQ  (\cI - \cT) \bb_2 \big| 
\cT \bb_1 - \cT \bb_2 \big \rangle \ge 0.
\]
Adding $\|(\cI - \cT)\bb_1 - (\cI - \cT) \bb_2 \|_\cQ^2$ on both sides, we obtain:
\[
\big \langle \cQ (\cI - \cT) \bb_1 -\cQ  (\cI - \cT) \bb_2 \big|   \bb_1 -  \bb_2 \big \rangle  \ge 
\big\| (\cI - \cT)\bb_1 - (\cI - \cT) \bb_2 \big\|_\cQ^2,
\]
which leads to (v), noting that  $\big \langle \cQ (\cI - \cT) \bb_1 -\cQ  (\cI - \cT) \bb_2 \big|   \bb_1 -  \bb_2 \big \rangle = \big \langle  (\cI - \cT) \bb_1 -  (\cI - \cT) \bb_2 \big| \cQ^*(  \bb_1 -  \bb_2) \big \rangle$. 
\hfill 
\end{proof}

\begin{remark}
As mentioned in Sect. \ref{sec_mot}, Lemma \ref{l_eq} is very useful for proving the convergence and asymptotic regularity of the generalized proximal point algorithm with a non-trivial relaxation step, where $\cQ$ is not  self-adjoint and PD. More specifically, the new metric becomes $\cS=\cQ \cM^{-1}$ instead of $\cQ$, where $\cM$ is a linear relaxation operator, such that $\cS$ is self-adjoint and PD \cite{hbs_siam_2012,fxue_gopt}.
\end{remark}

\subsection{Averagedness, cocoerciveness and Lipschitz continuity}
\label{sec_ave}
The following results  extend \cite[Proposition 4.35, Remark 4.34, Remark 4.37, Proposition 4.39, Proposition 4.40]{plc_book} to arbitrary metric $\cQ $, which build the connections of $\cQ$--based 1--Lipschitz $\alpha$--averagedness (i.e.  $\cF^\cQ_{1,\alpha}$)  to other concepts.
\begin{lemma} \label{l_average}
Let $\cT: D \mapsto \cH$, then, the following hold.

{\rm (i)}  $\cT \in \cF^\cQ_{1,\alpha}$ with $\alpha\in \ ]0,1[$, if and only if:
\[
\big\| \cT\bb_1-\cT\bb_2 \big\|_\cQ^2 +  
\frac{1-\alpha}{\alpha}  \big\|(\cI - \cT) \bb_1 - (\cI - \cT) \bb_2  \big\|_\cQ^2 \le 
 \big\| \bb_1 - \bb_2  \big\|_\cQ^2.
\]

{\rm (ii)}  $\cT$ is $\cQ$--firmly nonexpansive, if and only if  $\cT \in \cF^\cQ_{1,\frac{1}{2}}$.

{\rm (iii)}  If  $\cT \in \cF^\cQ_{1,\alpha}$  with $\cQ \in \cM^+$ and $\alpha \in \ ]0,\frac{1}{2}]$, then $\cT$ is $\cQ$--firmly nonexpansive.

{\rm (iv)} Let $\alpha \in\  ]0,1[$, $\gamma \in \ ]0, \frac{1}{\alpha} [$, then,  $\cT \in \cF^\cQ_{1,\alpha}$, if and only if $(1-\gamma)\cI + \gamma \cT \in \cF^\cQ_{1,\gamma \alpha}$. 

{\rm (v)}  $\cT$ is $\cQ$--based $\beta$--cocoercive with $\cQ \in \cM_\cS$, if and only if $\beta \cT \in \cF^\cQ_{1,\frac{1}{2} }$.

{\rm (vi)}  Let $\cT$ be $\cQ$--based $\beta$--cocoercive with $\cQ \in \cM_\cS$. If $\gamma \in \ ]0, 2\beta[$, then, $\cI -\gamma\cT   \in \cF^\cQ_{1,\frac{\gamma}{2\beta} }$.
\end{lemma}
\begin{proof}
(i) \cite[Proposition 4.35]{plc_book}--(iii) and Lemma \ref{l_id_1};

(ii) \cite[Remark 4.34]{plc_book}--(iii);

(iii) \cite[Remark 4.37]{plc_book};

(iv) \cite[Proposition 4.40]{plc_book}.

(v) Lemma \ref{l_average}--(ii), Lemma \ref{l_part}--(iii) and Definition \ref{def_nonexpansive}--(v).

(vi) Lemma \ref{l_average}--(v) and \cite[Proposition 4.39]{plc_book}.
\hfill 
\end{proof}

The following theorem, as a main result of this part, collects the key results of  $\cF^\cQ_{\xi,\alpha}$.
\begin{theorem} \label{t_lip}
Let  $\cT \in \cF^\cQ_{\xi,\alpha}$ with  $\xi \in \ ]0,+\infty[$ and $\alpha\in \ ]0,1[$. Then, the following hold.

{\rm (i)}  $\cT$ satisfies:
\be \label{xi_alpha}
\big\| \cT\bb_1-\cT\bb_2 \big\|_\cQ^2 \le   
(1-\alpha + \alpha \xi^2) \big\| \bb_1 - \bb_2  \big\|_\cQ^2
 - \frac{1-\alpha}{\alpha}  \big\|(\cI - \cT) \bb_1 - (\cI - \cT) \bb_2  \big\|_\cQ^2.
\ee

{\rm (ii)} If $\cQ \in \cM^+$, $0< \xi \le \min\{ \frac{1-\alpha}{\alpha}, 1\}$, then $\cT$ is $\cQ$--firmly nonexpansive.

{\rm (iii)} If $\cQ \in \cM_\cS^+$,   $ \xi \in \ ]0,  \frac{1-\alpha}{\alpha} ]$, then $\cT$ is $\cQ$--based $\beta$--cocoercive, with $\beta = \frac{1}{2}  \big( 1 + 
\frac{1}{1 - \alpha + \alpha\xi^2} \big)$.

{\rm (iv)}  $\cI - \gamma \cT  \in \cF^\cQ_{\frac{\alpha \xi}  {1-\alpha}, \gamma(1-\alpha) } $, if  $\gamma \in \ ]0, \frac{1}{1-\alpha} [$.

{\rm (v)} If  $\cQ \in \cM^+$, $\gamma \in \ ]0,   \frac{1}{ 1-\alpha } [ $,  $\xi \le \min\{ \frac{1}{\alpha \gamma} - \frac{1-\alpha}{\alpha}, \frac{1-\alpha}{\alpha} \}$,  then $\cI - \gamma \cT$ is  $\cQ$--firmly nonexpansive.

{\rm (vi)} If  $\cQ \in \cM_\cS^+$, $\gamma \in\  ]0,   \frac{1}{ 1-\alpha } [ $,  $\xi \in \ ]0,  \frac{1}{\alpha \gamma} - \frac{1-\alpha}{\alpha} ] $,  then $\cI - \gamma \cT$ is  $\cQ$--based $\beta$--cocoercive with $\beta = \frac{1} {2} (1+ \frac{1-\alpha}
{1-\alpha-\gamma (1-\alpha)^2 +\gamma \alpha^2 \xi^2})$.

{\rm (vii)} The reflected operator of $\cT$ follows $2\cT - \cI \in \cF^\cQ_{\xi, 2 \alpha} $,   if $\alpha\in \ ]0, \frac{1}{2} [$.
\end{theorem}
\begin{proof}
(i) By Definition \ref{def_lip}, there exists a $\cQ$--based $\xi$--Lipschitz continuous operator $\cK: D \mapsto \cH$, such that $\cT = (1-\alpha) \cI + \alpha \cK$, and thus, $\cK = \frac{1}{\alpha} \cT + (1-\frac{1}{\alpha}) \cI$. By Lemma \ref{l_id_1},  we have:
\begin{eqnarray}
&& \big\|\cK \bb_1 - \cK \bb_2\big\|_\cQ^2 
\nonumber\\
& = & (1-\frac{1}{\alpha}) \big\| \bb_1 -  \bb_2\big\|_\cQ^2 
+ \frac{1}{\alpha} \big\|\cT \bb_1 - \cT \bb_2\big\|_\cQ^2 
+ \frac{1-\alpha}{\alpha^2} 
\big\| (\cI-\cT) \bb_1 - (\cI - \cT) \bb_2\big\|_\cQ^2 
\nonumber\\
& \le & \xi^2 \big\| \bb_1 -  \bb_2\big\|_\cQ^2,
\quad \text{(by Lipschitz continuity of $\cK$)}
\nonumber 
\end{eqnarray}
which yields the desired inequality, after simple rearrangements.

\vskip.2cm
(ii) If $\cQ \in \cM^+$, to ensure that $\cT$ is $\cQ$--firmly nonexpansive, we need to let  $1-\alpha +\alpha\xi^2 \le 1$ and $\frac{1-\alpha}{\alpha} \ge 1$,   by \eqref{xi_alpha} and Definition \ref{def_nonexpansive}--(iv). It 
 yields $\xi \in \ ]0,1]$ and $\alpha \in \ ]0, \frac{1}{2}]$.

On the other hand, rewrite \eqref{xi_alpha} as:
\be \label{y11}
\frac {\alpha} {1-\alpha} \big\| \cT\bb_1-\cT\bb_2 \big\|_\cQ^2 \le   \frac {\alpha} {1-\alpha} 
(1-\alpha + \alpha \xi^2) \big\| \bb_1 - \bb_2  \big\|_\cQ^2
 -  \big\|(\cI - \cT) \bb_1 - (\cI - \cT) \bb_2  \big\|_\cQ^2. 
\ee
The firm nonexpansiveness of $\cT$ requires
 $\frac {\alpha} {1-\alpha}  \ge 1$  and $ \frac {\alpha} {1-\alpha}  (1-\alpha + \alpha \xi^2) \le 1$, i.e. $ \xi \in\  ]0,  \frac{1-\alpha}{\alpha} ] $ and $\alpha \in [ \frac{1}{2}, 1[$.
Finally, combining both conditions yields $0 < \xi \le  \min\{ \frac{1-\alpha}{\alpha}, 1\}$.

\vskip.2cm
(iii) If $\cQ \in \cM_\cS$,  expanding $\|(\cI- \cT) \bb_1 - (\cI - \cT)\bb_2 \|_\cQ^2$,  \eqref{xi_alpha} is equivalent to:
\[
\frac{2(1-\alpha)}{\alpha}   \big\langle  \cQ(\bb_1 - \bb_2) \big| 
 \cT \bb_1 -   \cT  \bb_2  \big\rangle   \ge 
\frac{1}{\alpha} \big\| \cT\bb_1-\cT\bb_2 \big\|_\cQ^2
- (2 - \alpha  -\frac{1}{\alpha} + \alpha \xi^2) \big\| \bb_1 - \bb_2  \big\|_\cQ^2,
\] 
which yields:
\be \label{z1}
 \big\langle \cQ (\bb_1 - \bb_2) \big| 
 \cT \bb_1 -   \cT  \bb_2  \big\rangle   \ge 
\frac {1} {2(1-\alpha)}  \big\| \cT\bb_1-\cT\bb_2 \big\|_\cQ^2,
\ee 
if $\cQ \in \cM^+$ and  $ 2 - \alpha  -\frac{1}{\alpha} + \alpha \xi^2 \le 0 $, i.e. $\xi \le \frac{1-\alpha}{\alpha}$.

On the other hand,  if $\cQ \in \cM_\cS^+$,   \eqref{y11} becomes:
\begin{eqnarray}
&& \frac {\alpha} {1-\alpha} \big\| \cT\bb_1-\cT\bb_2 \big\|_\cQ^2 
\nonumber \\
 & \le &  \frac {\alpha (1-\alpha + \alpha \xi^2)} {1-\alpha} 
 \big\| \bb_1 - \bb_2  \big\|_\cQ^2
 -  \big\|(\cI - \cT) \bb_1 - (\cI - \cT) \bb_2  \big\|_\cQ^2 
\nonumber \\
 & = &  \frac {\alpha (1-\alpha + \alpha \xi^2)} {1-\alpha} 
 \big\|(\cI - \cT) \bb_1 - (\cI - \cT) \bb_2
 + \cT \bb_1 - \cT \bb_2  \big\|_\cQ^2
 -  \big\|(\cI - \cT) \bb_1 - (\cI - \cT) \bb_2  \big\|_\cQ^2 
\nonumber \\
 & = & \Big( \frac {\alpha (1-\alpha + \alpha \xi^2)} {1-\alpha} -1 \Big)
 \big\|(\cI - \cT) \bb_1 - (\cI - \cT) \bb_2 \big\|_\cQ^2
-   \frac {\alpha (1-\alpha + \alpha \xi^2)} {1-\alpha} 
 \big\|  \cT  \bb_1 -    \cT  \bb_2 \big\|_\cQ^2
\nonumber \\
& +  & 2 \frac {\alpha (1-\alpha + \alpha \xi^2)} {1-\alpha}
\big\langle \cQ (\bb_1 -  \bb_2) \big| 
 \cT \bb_1 - \cT \bb_2  \big\rangle.
 \nonumber 
\end{eqnarray}
If $ \frac {\alpha (1-\alpha + \alpha \xi^2)} {1-\alpha}  \le 1$, i.e. $\xi \le \frac{1-\alpha}{\alpha}$, it   yields:
\[
2 \frac {\alpha (1-\alpha + \alpha \xi^2)} {1-\alpha}
\big\langle \cQ (\bb_1 -  \bb_2) \big| 
 \cT \bb_1 - \cT \bb_2  \big\rangle  \ge 
 \Big( \frac {\alpha} {1-\alpha}
 + \frac {\alpha (1-\alpha + \alpha \xi^2)} {1-\alpha} \Big)
  \big\| \cT\bb_1-\cT\bb_2 \big\|_\cQ^2,
\]
i.e.
\be \label{z2}
\big\langle \cQ (\bb_1 -  \bb_2) \big| 
 \cT \bb_1 - \cT \bb_2  \big\rangle  \ge 
\frac{1}{2} \Big(1 + \frac{1}{1 - \alpha + \alpha\xi^2} \Big) 
  \big\| \cT\bb_1-\cT\bb_2 \big\|_\cQ^2 .
\ee

Finally, (iii) follows by comparing  \eqref{z1} with \eqref{z2}, and noting that: 
$\frac{1}{1-\alpha} \le 1 + 
\frac{1}{1 - \alpha + \alpha\xi^2}$, if $\xi \le \frac{1-\alpha}{\alpha}$.

\vskip.2cm
(iv) Expanding  $\|(\cI- \cT) \bb_1 - (\cI - \cT)\bb_2 \|_\cQ^2$,  \eqref{xi_alpha} is equivalent to:
\[
 \big\langle (\cQ+\cQ^*) ( \bb_1- \bb_2) \big|  \cT \bb_1   - \cT  \bb_2  \big\rangle
 \ge  \frac{1}{ 1-\alpha } \big\| \cT\bb_1-\cT\bb_2 \big\|_\cQ^2  - \frac{ \alpha^2 \xi^2 - (1 - \alpha)^2  }
{ 1-\alpha }  \big\| \bb_1 - \bb_2  \big\|_\cQ^2.
\]
Then, we have:
\begin{eqnarray}
&& \big\|(\cI - \gamma \cT) \bb_1 - (\cI - \gamma \cT) \bb_2\big\|_\cQ^2 
\nonumber\\
&=&    \big\|  \bb_1 -  \bb_2 \big\|_\cQ^2
-  \gamma \big\langle (\cQ+\cQ^*) (\bb_1 - \bb_2) \big| 
\cT \bb_1 - \cT \bb_2 \big\rangle  
+ \big\| \gamma  \cT   \bb_1 - \gamma   \cT  \bb_2 \big\|_\cQ^2
\nonumber\\
& \le &    \big\|  \bb_1 -  \bb_2 \big\|_\cQ^2
- \gamma  \frac{1}{1-\alpha } \big\| \cT\bb_1-\cT\bb_2 \big\|_\cQ^2  +  \gamma  \frac{ \alpha^2 \xi^2 - (1 - \alpha)^2  }
{ 1-\alpha }  \big\| \bb_1 - \bb_2  \big\|_\cQ^2 
+ \big\| \gamma  \cT   \bb_1 - \gamma   \cT  \bb_2 \big\|_\cQ^2
\nonumber\\
&=&  \Big(  1+ \gamma 
 \frac{ \alpha^2 \xi^2 - (1 - \alpha)^2  }  { 1-\alpha } \Big)
  \big\|  \bb_1 -  \bb_2 \big\|_\cQ^2
- \Big(  \frac{1}{\gamma (1-\alpha)} -1\Big)
 \big\|\gamma \cT\bb_1 - \gamma \cT\bb_2 \big\|_\cQ^2.   
\nonumber 
\end{eqnarray}
Let  $\cI - \gamma \cT \in \cF^\cQ_{\xi', \alpha'}$, then, by  \eqref{xi_alpha}, we have $\frac{1-\alpha'}{\alpha'} =  \frac{1}{\gamma (1-\alpha)} -1 $ and $1-\alpha'+\alpha' \xi'^2 = 1+ \gamma 
 \frac{ \alpha^2 \xi^2 - (1 - \alpha)^2  }  { 1-\alpha }$, which yields $\alpha' = \gamma(1-\alpha)$ and $\xi' = \frac{\alpha \xi}  {1-\alpha} $.
 
\vskip.2cm
(v) Theorem \ref{t_lip}--(ii) and (iv).

\vskip.2cm
(vi) Theorem \ref{t_lip}--(iii) and (iv).

\vskip.2cm
(vii)  We deduce that:
\begin{eqnarray}
&& \big\|(2\cT - \cI) \bb_1 - (2\cT - \cI) \bb_2\big\|_\cQ^2 
\nonumber\\
&=&   2  \big\|\cT  \bb_1 - \cT \bb_2 \big\|_\cQ^2
-\big\| \bb_1 -\bb_2 \big\|_\cQ^2 
+2 \big\|(\cI-\cT)  \bb_1 - (\cI-\cT) \bb_2 \big\|_\cQ^2
\ \text{(by Lemma \ref{l_id_1})}
\nonumber\\
& \le & 2(1-\alpha + \alpha \xi^2) \big\| \bb_1 -  \bb_2\big\|_\cQ^2 
- \frac{2(1-\alpha)}{\alpha} \big\| (\cI-\cT) \bb_1 - (\cI - \cT) \bb_2\big\|_\cQ^2 
\nonumber\\
& - & \big\| \bb_1 -\bb_2 \big\|_\cQ^2 
+2 \big\|(\cI-\cT)  \bb_1 - (\cI-\cT) \bb_2 \big\|_\cQ^2
\quad  \text{(by \eqref{xi_alpha} )}
\nonumber\\
& = & (1-2\alpha + 2\alpha \xi^2) \big\| \bb_1 -  \bb_2\big\|_\cQ^2 
- \frac{1-2\alpha} {2\alpha} \big\| (2\cI-2\cT) \bb_1 - 
(2\cI - 2\cT) \bb_2\big\|_\cQ^2.
\nonumber 
\end{eqnarray}
Let $2\cT - \cI \in \cF^\cQ_{\xi',  \alpha'}$. Thus, we have $\frac{1-\alpha'}{\alpha'} = \frac{1-2\alpha}{2\alpha}  $ and $1-\alpha'+\alpha' \xi'^2 = 1-2\alpha + 2\alpha \xi^2$,  i.e. $\alpha' =2  \alpha $, and $\xi' =\xi$. \hfill 
\end{proof}

Two corollaries follow from Theorem \ref{t_lip}.
\begin{corollary} \label{c_lip_1} 
{\rm [Further results of Theorem \ref{t_lip}--(iii)]}
Let  $\cT \in \cF^\cQ_{\xi,\alpha}$ with  $\xi \in \  ]0,+\infty[$ and $\alpha\in \ ]0,1[$, then, the following hold.

{\rm (i)} If $\xi \le \min\{\frac{1-\alpha}{\alpha}, 1\}$, then, $\cT$ is $\cQ$--based $\beta$--cocoercive with $\beta \in [1, +\infty [ $,  strongly $\alpha$--averaged, and $\cQ$--firmly nonexpansive.

{\rm (ii)} If $\alpha \in \ ]0, \frac{1}{2} [ $, $\xi \in\  ]1, \frac{1-\alpha}{\alpha} ]$, then, $\cT$ is $\cQ$--based $\beta$--cocoercive with $\beta \in \ ]0, 1[ $, and weakly  $\alpha$--averaged.
\end{corollary}
\begin{proof}
By Theorem \ref{t_lip}-(iii), $\cT$ is  $\beta$-cocoercive, with $\beta = \frac{1}{2}  \big( 1 + 
\frac{1}{1 - \alpha + \alpha\xi^2} \big)$. The proof is completed by comparing $\beta$ with 1.
\hfill 
\end{proof}

\begin{corollary} \label{c_lip_2}
{\rm [Further results of Theorem \ref{t_lip}--(vi)]} 
Let  $\cT \in \cF^\cQ_{\xi,\alpha}$ with  $\xi \in ]0,+\infty[$ and $\alpha\in \ ]0,1[$. If $\gamma \in \ ]0, \frac{1}{1-\alpha} [$,  then, the following hold.

{\rm (i)} If $\xi \le \min\{\frac{1-\alpha}{\alpha}, 
\frac{1}{\gamma \alpha} - \frac{1-\alpha}{\alpha} \}$, then, 
$\cI - \gamma \cT$ is $\cQ$--based $\beta$--cocoercive with $\beta \in [ 1, +\infty [ $,  strongly $\alpha$--averaged, and $\cQ$--firmly nonexpansive.

{\rm (ii)} If $\gamma \in \ ]0, \frac{1}{2(1-\alpha)} [ $, $\xi \in\  ] \frac{1-\alpha}{\alpha}, \frac{1}{\gamma \alpha} - \frac{1-\alpha}{\alpha} ]$, then, $\cI -\gamma \cT$ is $\cQ$--based $\beta$--cocoercive with $\beta \in \  ]0, 1[ $, and weakly  $\alpha$--averaged.
\end{corollary}
\begin{proof}
The proof is completed by comparing $\beta$ in Theorem \ref{t_lip}--(vi)  with 1. \hfill 
\end{proof}

\begin{lemma} \label{l_cocoercive}
Let the operator $\cT: D \mapsto \cH$ be $\cQ$--based $\beta$--cocoercive with $\cQ \in \cM_\cS$ and $\beta \in \ ]\frac{1}{2}, +\infty [$. Then, the following hold.

{\rm (i)} If $\gamma \in \ ]0, 2\beta [$, then, $\cT \in \cF^\cQ_{\frac{1}{2\beta-1}, 
1-\frac{1}{2\beta} }$,  $\cI - \gamma \cT \in \cF^\cQ_{1,\frac{\gamma} {2\beta}}$.

{\rm (ii)} If $\cQ \in \cM_\cS^+$,  $\gamma \in \ ] 0,  \beta]$,  $\cI - \gamma \cT$ is $\cQ$--based 1--cocoercive (i.e. $\cQ$--partly nonexpansive).
\end{lemma}
\begin{proof}
(i)  If $\cQ \in \cM_\cS$, we have:
\begin{eqnarray}
&& \big\|(\cI -  \cT) \bb_1 - (\cI - \cT) \bb_2\big\|_\cQ^2 
\nonumber\\
&=&    \big\|  \bb_1 -  \bb_2 \big\|_\cQ^2
- 2 \big\langle \cQ(\bb_1 - \bb_2) \big| 
\cT \bb_1 - \cT \bb_2 \big\rangle  
+ \big\|  \cT   \bb_1 -  \cT  \bb_2 \big\|_\cQ^2
\nonumber\\
& \le & \big\|  \bb_1 -  \bb_2 \big\|_\cQ^2
- 2\beta \big\| \cT  \bb_1 - \cT \bb_2 \big\|_\cQ^2 
+ \big\|  \cT   \bb_1 -  \cT  \bb_2 \big\|_\cQ^2
\nonumber\\
&=&  \big\|  \bb_1 -  \bb_2 \big\|_\cQ^2
- ( 2\beta -1) \big\| \cT  \bb_1 - \cT \bb_2 \big\|_\cQ^2,
\nonumber 
\end{eqnarray}
which yields:
\[
\big\| \cT  \bb_1 - \cT \bb_2 \big\|_\cQ^2 
\le \frac{1}{2\beta -1}  \big\|  \bb_1 -  \bb_2 \big\|_\cQ^2
- \frac{1}{2\beta -1} 
\big\|(\cI -  \cT) \bb_1 - (\cI - \cT) \bb_2\big\|_\cQ^2.
\]
Thus,  if $\cT \in \cF^\cQ_{\xi', \alpha'}$,  by \eqref{xi_alpha},  we have: $\frac{1-\alpha'}{\alpha'} = \frac{1}{2\beta-1}$  and $1-\alpha'+\alpha' \xi'^2 = \frac{1}{2\beta-1}$, i.e. 
$\alpha'=1-\frac{1}{2\beta}$ and $\xi'=\frac{1}{2\beta-1}$. 

$\cI - \gamma \cT$ follows from Theorem \ref{t_lip}-(iv).

\vskip.2cm
(ii) Theorem \ref{t_lip}-(iii) and Lemma \ref{l_cocoercive}-(i).
\hfill 
\end{proof}

\vskip.2cm
Part of the results in Theorem \ref{t_lip}, Corollary \ref{c_lip_1} and Corollary \ref{c_lip_2} is summarized in Fig.1, where FNE stands for `$\cQ$--firmly nonexpansive'.  We can see that  
 $\cT \in \cF^\cQ_{\xi,\alpha}$ could be $\cQ$--firmly nonexpansive for $\alpha> 1/2$, at the expense of stricter condition on the Lipschitz constant $\xi \le \frac{1-\alpha}{\alpha} <1 $.
\begin{figure} [H]
\hspace*{-.3cm}
\scalebox{0.7} {
\begin{tikzpicture}[scale=0.97]
\draw node[below] at (0,0) {\large $\cT \in \cF^\cQ_{\xi,\alpha}$};
\draw [line width=.8pt]  (1, -.4) -- (3, -.4);
\draw node[below] at (2, 0.3) { $\xi \le \frac{1-\alpha}{\alpha}$};
\draw [line width=.8pt]  (3, .4) -- (3, -1.3);
\draw [->, line width=.8pt]  (3, .4) -- (4.2, .4);
\draw node[below] at (3.5, 1.1) { $\alpha \ge \frac{1}{2}$};
\draw node[below] at (9.2, .7) {\large $\beta$--cocoercive with $\beta \ge 1$,  FNE, strongly $\alpha$--averaged};

\draw [->, line width=.8pt]  (3, -1.3) -- (4.2, -1.3);
\draw node[below] at (3.5, -1.3) { $\alpha < \frac{1}{2}$};
\draw node[below] at (5.5, -1) {\large $\beta$--cocoercive};
\draw [line width=.8pt]  (6.8, -1.3) -- (7.5, -1.3);
\draw [line width=.8pt]  (7.5, -2) -- (7.5, -.5);
\draw [->, line width=.8pt]  (7.5, -.5) -- (9.1, -.5);
\draw node[below] at (8.3, -.5) { $\xi \le  1$};
\draw node[below] at (12.3, -.2) {\large $\beta \ge 1$, FNE, strongly $\alpha$--averaged};
\draw [->, line width=.8pt]  (7.5, -2) -- (9.1, -2);
\draw node[below] at (8.3, -1.4) { $\xi > 1$};
\draw node[below] at (12.7, -1.6) {\large $\beta < 1$, non--FNE, weakly $\alpha$--averaged};

\draw [->, line width=.8pt]  (0, -0.6) -- (0, -4.5);
\draw node[below] at (-.3, -4.6) {\large $\cI - \gamma \cT \in \cF^\cQ_{\frac{\alpha \xi} {1-\alpha}, \gamma(1-\alpha)}$};
\draw [line width=.8pt]  (1, -5) -- (3, -5);
\draw node[below] at (1.7, -4.3) { $\xi \le \frac{1}{\gamma \alpha} - \frac{1-\alpha} {\alpha}$};

\draw [line width=.8pt]  (3, -4) -- (3, -6);
\draw [->, line width=.8pt]  (3, -4) -- (5, -4);
\draw node[below] at (4, -3.3) { $\gamma < \frac{1}{2(1-\alpha)}$};
\draw node[below] at (6.3, -3.7) {\large  $\beta$--cocoercive};
\draw [ line width=.8pt]  (7.6, -4) -- (8.2, -4);
\draw [line width=.8pt]  (8.2, -5) -- (8.2, -3);
\draw [->, line width=.8pt]  (8.2, -3) -- (9.8, -3);
\draw node[below] at (9, -3) {\large $\xi \le \frac{1-\alpha}{ \alpha }$};
\draw node[below] at (13, -2.7) {\large  $\beta \ge 1$, FNE, strongly $\alpha$--averaged};

\draw [->, line width=.8pt]  (8.2, -5) -- (9.8, -5);
\draw node[below] at (9, -4.3) { $  \xi > \frac{1-\alpha}{ \alpha }$};
\draw node[below] at (13.25, -4.7) {\large  $\beta < 1$, non--FNE, weakly $\alpha$--averaged};

\draw [->, line width=.8pt]  (3, -6) -- (5, -6);
\draw node[below] at (4, -5.3) { $\gamma \ge  \frac{1}{2(1-\alpha)} $};

\draw node[below] at (10, -5.7) {\large $\beta$--cocoercive with $\beta \ge 1$, FNE, strongly $\alpha$--averaged};

\draw [->, line width=.8pt]  (0, .1) -- (0, 1.5);
\draw node[below] at (0, 2) {\large $2\cT - \cI \in \cF^\cQ_{\xi, 2\alpha}$};
\draw node[below] at (0, 1) {\large if  $\alpha < 1/2$};

\draw node[below] at (0, -2) {\large if  $\gamma < \frac{1}{1-\alpha}$};
\end{tikzpicture} }
\vskip-.25cm
\caption{The properties of $\cF^\cQ_{\xi,\alpha}$, under the condition of $\cQ \in \cM_\cS^+$.  }
\vskip-.25cm
\end{figure}

\begin{remark}
Many results in Sect. \ref{sec_ave} are useful for analyzing the relaxed version of fixed point iterations (see Sect. \ref{sec_km}), and proving linear convergence under stronger conditions (e.g. the case of $\xi \in\ ]0,1[$ in Proposition \ref{p_dist} and Corollary \ref{c_banach}).
\end{remark}

\section{The associated fixed-point iterations}
\label{sec_fixed}
\subsection{Assumptions}
The convergence of the fixed-point iterations associated with the non-degenerate  metric-based nonexpansive operator $\cT$ has been well understood in literature, see \cite{plc_book,hhb_resolvent,fxue_gopt,plc_vu} for some typical results. In this sequel, we focus on the degenerate case only. More specifically, we make the following assumption on $\cQ$:
\begin{assumption} \label{assume_1}
$\cQ \in \cM_\cS^+$, such that $\kersf \cQ \backslash \{0\} \ne \emptyset$.
\end{assumption}

\begin{remark}
Assumption \ref{assume_1} implies that $\cQ$ has a non-trivial null space in the degenerate case, which is the focus of our discussion.
%
\end{remark}


We also make several assumptions on $\cT$:
\begin{assumption} \label{assume_2}

{\rm (i)} $\cT \in \cF_{\xi,\alpha}^\cQ$  with $\alpha \in \ ]0,1[$ and $\xi \in \ ]0, 1]$, i.e. $\cT$ is $\cQ$--strongly averaged (see Definition \ref{def_lip}).

{\rm (ii)} The set $\Fix  \cT := \{x\in D | x = \cT x\}$ is non-empty.

{\rm (iii)} $\cT: D\mapsto \cH$ is demiclosed, where $D$ is a nonempty weakly sequentially closed subset of $\cH$.

{\rm (iv)} $\cT$ satisfies $\|\cT x_1  - \cT x_2\| \le L \|x_1-x_2\|_\cQ$ for some constant $L$.

\end{assumption}

\begin{remark}
Assumption \ref{assume_2}-(i) is a conventional condition, commonly used in classical results, to guarantee the basic nonexpansiveness. In (ii), the existence of fixed point set of $\cT$ is a subtle assumption, which, however, is reasonable in many applications, as shown in Sect. \ref{sec_appl}.  

(iii) implies that $\gra \cT$ is sequentially closed in $\cH_\text{\rm weak} \times \cH_\text{\rm strong}$, which is useful to prove the convergence.
Generally speaking, (iv) is a rather restrictive condition, which indicates that all of the useful information of $\cT x$ lies in $\ran \cQ$, instead of the whole space $\cH$. We will see that (iii) and (iv) are essential to prove the boundedness of $\{x^k\}_{k\in\N}$ and the strong convergence of $x^k - \cT x^k \rightarrow 0$ in $\cH$, as $k\rightarrow \infty$. In addition, many degenerate metric resolvents (e.g. discussed in Sect. \ref{sec_appl}) satisfy  this  rigid requirement.
\end{remark}

\vskip.2cm 
We define a fixed point as $x^\star \in  \Fix  \cT $. The  $\cQ$--based solution distance and  sequential error  of the $k$--th iterate are defined by $\|\bb^k - x^\star\|_\cQ$ and $\|x^{k+1} - x^k\|_\cQ$, respectively.
The $\cQ$--based sequential error is closely related to {\it $\cQ$--asymptotic regularity}, which is an extended version of {\it asymptotically regular} \cite{koh,bau_2003}.
\begin{definition} 
{\rm A mapping $\cT: D \mapsto \cH $ is {\it $\cQ$--asymptotically regular}, if $\| \cT^k \bb - \cT^{k+1} \bb \|_\cQ \rightarrow 0$}, as $k \rightarrow \infty$, $\forall \bb \in D$. Here, $\cT^k$ is defined as: $\cT^k := \underbrace{\cT \circ \cdots \circ \cT}_\text{$k$ times}$.
\end{definition}

Clearly, if $\cT$ is $\cQ$--asymptotically regular, the $\cQ$--based sequential error vanishes, as $k \rightarrow \infty$. However, it does not necessarily yield the strong convergence of $x^k-x^{k+1} \rightarrow 0$, due to the degeneracy of $\cQ$. Nonetheless, if Assumption \ref{assume_2}-(iv) is taken into account, we can obtain the following important observation:
\begin{fact} \label{f_reg}
Under Assumption \ref{assume_2}-(iv), if $\cT$ is $\cQ$--asymptotically regular, then it is also asymptotically regular.
\end{fact}
\begin{proof}
By Assumption \ref{assume_2}-(iv), we have $\|\cT^{k+1} x - \cT^{k+2} x\| \le L \|\cT^k x - \cT^{k+1} x\|_\cQ$, which yields the desired result by taking $k \rightarrow \infty$.
\hfill 
\end{proof}

\subsection{Banach-Picard iteration}
Considering the scheme
\be \label{banach}
\bb^{k+1} := \cT \bb^k,
\ee
the properties of  metric-based distances of \eqref{banach} are given as follows. 
\begin{proposition}[Convergence in $\ran \cQ$] \label{p_dist}
Let $\bb^0\in D$, $\{\bb^k\}_{k \in \N}$ be a sequence generated by \eqref{banach}.  Denote $\nu:= 1-\alpha + \alpha \xi^2$. Under Assumptions \ref{assume_1} and    \ref{assume_2}-(i-ii), the following hold.

{\rm (i)} $\cT$ is $\cQ$--asymptotically regular.

{\rm (ii) [Sequential error]}   $\|\bb^{k+1 } -\bb^{k} \|_\cQ$ has the pointwise sublinear convergence rate of $\cO(1/\sqrt{k})$:
\[
\big\|\bb^{k +1} -\bb^{k} \big\|_\cQ
\le \frac{1}{\sqrt{k+1}}   \sqrt{ \frac{\alpha}{1-\alpha} }
\big\|\bb^{0} -\bb^\star \big\|_\cQ, \quad
\forall k \in \N.
\]

{\rm (iii) [$q$--linear convergence]} If $\xi \in \ ]0, 1[$,  both $\|\bb^{k} - x^\star \|_\cQ$ and $\|\bb^{k} -\bb^{k+1} \|_\cQ$  are $q$--linearly convergent with the rate of $\sqrt{\nu} $.

{\rm (iv) [$r$--linear convergence]}  If $\alpha \in \big] 1-\frac{1} {\sqrt{ 2} } , 1 \big[$,
$\xi \in \Big] 0, \sqrt{ 1 -  \frac{ 2- \sqrt{2}} { 2  \alpha} } \Big] $, 
 $\big\| \bb^k - \bb^{k+1} \big\|_\cQ$  is globally $r$--linearly convergent w.r.t. $\big\| \bb^0 - x^\star \big\|_\cQ$:
\[
    \big\| \bb^k - \bb^{k+1} \big\|_\cQ
   \le\sqrt{  \frac{2\alpha (1-\nu) } {(1-\alpha) \nu }  }
    \cdot \nu^{\frac{k+1}{2} }  
\big\| \bb^0 - \bbstar \big\|_\cQ.
\]

{\rm (v) [Weak/strong convergence in $\ran \cQ$]} If $\xi=1$ or  $\xi \in \ ]0,1[$, there exists $x^\star \in \Fix \cT$, such that $\sqrt{\cQ} \bb^k \rightharpoonup \text{\rm or\ }  \rightarrow \sqrt{\cQ} x^\star$ respectively, as $k\rightarrow \infty$. 
\end{proposition}

\vskip.2cm
\begin{proof}
(i)  Taking $\bb_1 = \bb^k$ and $\bb_2 = x^\star \in \Fix  \cT $ in \eqref{xi_alpha}, we obtain:
\be \label{x12}
\big\|  \bb^{k+1} -\bbstar \big\|_\cQ^2
  \le   \nu \big\| \bb^k - \bbstar \big\|_\cQ^2
- \frac{1-\alpha}{\alpha} \big\|  \bb^k -  \bb^{k+1} \big\|_\cQ^2.
\ee
Noting $\nu \in \ ]1-\alpha, 1]$, and summing up \eqref{x12} from $k=0$ to $K$ yields:
\be \label{x34}
\sum_{k=0}^{K} \big\| \bb^k - \bb^{k+1} \big\|
_\cQ^2 \le  \frac{\alpha}{1-\alpha} 
\big\| \bb^0 - \bbstar \big\|_\cQ^2.
\ee 
Taking $K \rightarrow \infty$, we have: $\sum_{k=0}^{\infty} \big\| \bb^k - \bb^{k+1} \big\|_\cQ^2 \le \frac{\alpha}{1-\alpha}  \big\| \bb^{0} - \bb^\star \big\|_\cQ^2 < +\infty$, which   implies that  $\lim_{k\rightarrow \infty} \|\bb^k-\bb^{k+1}\|_\cQ =  0$.

\vskip.2cm
(ii)  Taking $\bb_1 = \bb^k$ and  $\bb_2 = \bb^{k+1}$ in \eqref{xi_alpha}, we have:
\be  \label{x33}
\big\|  \bb^{k+1} -\bb^{k+2} \big\|_\cQ^2
  \le    \nu \big\| \bb^k - \bb^{k+1}  \big\|_\cQ^2.
\ee
 $\nu \in \ ]1-\alpha, 1]$ implies that  $\| \bb^k - \bb^{k+1}  \|_\cQ$ is non--increasing. Then, (ii) follows from \eqref{x34}.

\vskip.2cm
(iii) If $\xi \in \  ]0,1[$, \eqref{x12} yields that
$\big\|  \bb^{k+1} - x^\star \big\|_\cQ^2
 \le \nu   \big\| \bb^k - x^\star \big\|_\cQ^2$, 
where $\nu \in\  ]1-\alpha, 1[$. The $\cQ$--based sequential error follows from   \eqref{x33}.

\vskip.2cm
(iv) If $\xi \in\ ]0,1[$, combining \eqref{x33} with \eqref{x34} yields:
\[
\Big( \nu^{-k} + \nu^{-(k-1)} + \cdots + 1\Big)
  \big\| \bb^k - \bb^{k+1} \big\|_\cQ^2
   \le  \frac{\alpha}{1-\alpha} 
\big\| \bb^0 - \bbstar \big\|_\cQ^2, 
\]
which leads to:
\[
    \big\| \bb^k - \bb^{k+1} \big\|_\cQ^2
   \le  \frac{\alpha (1-\nu) } {(1-\alpha) \nu } \cdot 
    \frac {1 }  {\nu^{-(k+1)} - 1} 
\big\| \bb^0 - \bbstar \big\|_\cQ^2.
\]

Clearly, if $ \nu^{-(k+1)} - 1   \ge \frac{1}{2} \nu^{-(k+1)}$, (i.e. $k \ge  \frac{\ln 2} {\ln (1/\nu) } - 1$), $\big\| \bb^k - \bb^{k+1} \big\|_\cQ^2$
is $r$--linearly convergent w.r.t. $\big\| \bb^0 - x^\star \big\|_\cQ^2$:
\[
    \big\| \bb^k - \bb^{k+1} \big\|_\cQ^2
   \le  \frac{2\alpha (1-\nu) } {(1-\alpha) \nu } \cdot 
    \nu^{ k+1  } \big\| \bb^0 - x^\star \big\|_\cQ^2. 
\]
Furthermore, if $ \frac{\ln 2} {\ln (1/ \nu ) } - 1 \le 1$, the $r$--linear convergence is globally valid for $\forall k \in \N$. This condition can be simplified as 
$\xi^2 \le  1 -  \frac{  2-\sqrt{2} } { 2  \alpha}  $.  

\vskip.2cm
(v) If $\xi=1$, the weak convergence of $\{\sqrt{\cQ} x^k\}_{k \in \N}$ is clear, by basic nonexpansive properties \cite[Theorem 5.14-(i), Example 5.18]{plc_book} of Fej\'{e}r monotonicity \cite[Proposition 5.4, Theorem 5.5]{plc_book}. 

In the case of $\xi \in \ ]0, 1[$, the linear convergence of $\{\sqrt{\cQ} x^k\}_{k \in \N}$ immediately follows by \cite[Theorem 5.12]{plc_book}.
\hfill 
\end{proof}

\begin{remark}
As emphasized above, one cannot conclude from Proposition \ref{p_dist} the convergence of  $\{x_k\}_{k\in\N}$ in the whole space, since the $\cQ$--metric distance does not infer anything about the projection of $x^k$ onto $\kersf \cQ$, which, however, has to be taken into account for the convergence in the whole space. 
\end{remark}

\vskip.2cm
The following theorem is a main result of this paper, which shows the convergence of $x^k$ in $\cH$ under additional Assumption \ref{assume_2}-(iii-iv). The proof adopts some techniques in \cite[Theorem 2.1]{fxue_gopt}.

\begin{theorem} [Weak convergence in $\cH$]  \label{t_banach}
Let $\bb^0\in D$, $\{\bb^k\}_{k \in \N}$ be a sequence generated by \eqref{banach}. Under Assumptions \ref{assume_1} and \ref{assume_2}, if $\xi=1$, 
then there exists $\bbstar \in \Fix \cT$, such that $\bb^k \rightharpoonup \bbstar$, as $k\rightarrow \infty$.
\end{theorem}

\begin{proof}
Following the reasoning of  the well-known Opial's lemma \cite{opial}\footnote{Refer to \cite[Lemma 2.47]{plc_book} or \cite[Lemma 2.1]{attouch_2001} for the Opial's argument.}, the proof is divided into 4 steps\footnote{This line of reasoning is very similar to Fej\'{e}r monotonicity, see \cite[Proposition 5.4,  Theorem 5.5]{plc_book} for example.}:

(i) for every $x^\star \in \Fix \cT $, $\lim_{k\rightarrow \infty} \| x^{k} - x^\star \|_\cQ  $ exists; 

(ii) the sequence $\{x^k\}_{k\in\N}$ is bounded;

(iii) if $x^{k_i} \rightharpoonup x^*$ weakly in $\cH$ for a subsequence ${k_i} \rightarrow \infty$, then $x^* \in \Fix \cT $;

(iv)  $\{x^k\}_{k\in\N}$ possesses at most one weak sequential cluster point in $\Fix \cT$.

\vskip.2cm
(i) \eqref{x12}  shows that   
 $\{ \| x^{k} - x^\star \|_\cQ \}_{k\in\N} $ is non-increasing, and bounded from below (always being non-negative), and thus, convergent, i.e. $\lim_{k\rightarrow \infty} \| x^{k} - x^\star \|_\cQ$ exists.  
 
\vskip.1cm
(ii) By Assumption \ref{assume_2}-(iv), we have:
\[
\big\|  x^{k+1} -  x^\star \big\| = \big\| \cT x^k -\cT x^\star \big\|  \le L \big\|x^k - x^\star\big\|_\cQ  
 \le L \big\|x^0 - x^\star\big\|_\cQ,\quad 
 \forall k\in\N,
\]
where the last inequality comes from (i). It implies that $\{x^k\}_{k\in \N}$ is bounded.

\vskip.1cm
(iii) Since $\{x^k\}_{k\in \N}$ is bounded, by \cite[Lemma 2.45]{plc_book}, $\{x^k\}_{k\in\N}$ has at least  one weak sequential cluster point, i.e. $\{x^k\}_{k\in\N}$ has a subsequence $\{x^{k_i}\}_{i\in\N}$ that weakly converges to a point $x^*$, denoted by $x^{k_i} \rightharpoonup x^*$, as $k_i \rightarrow \infty$. Our aim  is to show that $x^* \in \Fix \cT $, and more generally,  every weak sequential cluster point of  $\{x^k\}_{k\in\N}$ belongs to $\Fix \cT $. To this end, combining Assumption \ref{assume_2}-(iv) with the claim (i), the weakly convergent subsequence $\{x^{k_i}\}_{k\in\N}$  satisfies:
\[
\big\|x^{k_i+1}-x^{k_i+2} \big\| = \big\|\cT x^{k_i }- \cT x^{k_i+1} \big\| \le L \big\|x^{k_i}-x^{k_i+1} \big\|_\cQ \rightarrow 0, \quad \text{as\ } k_i\rightarrow \infty
\]
which shows that $x^{k_i}-x^{k_i+1}=x^{k_i}-\cT x^{k_i} \rightarrow 0$, as $k_i\rightarrow \infty$ (this is also Fact \ref{f_reg}). 
Since  $x^{k_i} \rightharpoonup x^*$ as $k_i\rightarrow \infty$, we conclude that  $ x^* - \cT x^* =0$ due to  the demiclosedness of $\cT$ (i.e. Assumption \ref{assume_2}-(iii) that implies that $\gra (\cI-\cT)$ is sequentially closed in $\cH_\text{weak} \times \cH_\text{strong}$).  Thus, for every weak sequential cluster point $x^*$ of  $\{x^k\}_{k\in\N}$, $x^* \in \Fix \cT $. 

\vskip.1cm
(iv) We need to show that   $\{x^k\}_{k\in\N}$ cannot have two distinct weak sequential cluster point in $\Fix \cT $. To this end, let $x_1^*, x_2^{*} \in \Fix  \cT $ be two cluster points of  $\{x^k\}_{k\in\N}$. Set $\eta_1 = \lim_{k\rightarrow \infty} \|x^k - x_1^*\|_\cQ$, and  $\eta_2 = \lim_{k\rightarrow \infty} \|x^k - x_2^*\|_\cQ$. Take a subsequence $\{x^{k_i} \}$ weakly converging to $x_1^*$, as $k_i \rightarrow \infty$. From the identity
\[
\big\| x^k-x_1^*\big\|_\cQ^2 - \big\|x^k-x_2^* \big\|_\cQ^2
=\big\|x_1^*-x_2^*\big\|_\cQ^2 +2 \big\langle \cQ
(x_1^*-x_2^*) \big|  x_2^* - x^k \big\rangle,
\]
we deduce that  $\eta_1 - \eta_2 =- \big\|x_1^*-x_2^*\big\|_\cQ^2$ by taking $k\rightarrow \infty$ on both sides. Similarly,  take a subsequence $\{x^{l_i} \}$ weakly converging to $x_2^*$, as $l_i \rightarrow \infty$, which yields that $\eta_1 - \eta_2 = \big\|x_1^* - x_2^*\big\|_\cQ^2$. Consequently, $\big\|x_1^*- x_2^*\big\|_\cQ=0$, i.e. $x_1^* - x_2^* \in \kersf \cQ$. Furthermore,  Assumption \ref{assume_2}-(iv) yields 
$\big\| \cT x_1^*  - \cT x_2^*\big\| \le L \big\| x_1^*  - x^*_2 \big\|_\cQ=0$,  which results in $\cT x_1^*=\cT x_2^*$, and thus, $x^*_1=x^*_2$, since $x_1^*,x_2^* \in \Fix \cT$. This shows the uniqueness of the weak sequential cluster point, denoted by $x^\star$.

\vskip.1cm
Finally, to summarize,  $\{x^k\}_{k\in\N}$  is bounded and possesses a unique weak sequential cluster point $x^\star \in \Fix \cT $. Then, the weak convergence is established by \cite[Lemma 2.46]{plc_book}.
\hfill 
\end{proof}

\vskip.1cm
It is much easier to prove the strong convergence of \eqref{banach} in the case of $\xi \in\ ]0,1[$.
\begin{corollary} [Strong convergence in $\cH$]  \label{c_banach}
Let $\bb^0\in D$, $\{\bb^k\}_{k \in \N}$ be a sequence generated by \eqref{banach}. Under Assumptions \ref{assume_1} and \ref{assume_2}, if $\xi \in\ ]0,1[$,  then there exists $\bbstar \in \Fix \cT$, such that $\bb^k \rightarrow \bbstar$, as $k\rightarrow \infty$.
\end{corollary}

\begin{proof}
If $\xi \in\ ]0,1[$, then $\nu <1$. For $x^\star \in \Fix \cT$, combining Assumption \ref{assume_2}-(iv) with \eqref{x12}, it yields:
\[
\big\|  \bb^{k+1} -\bbstar \big\|^2
=\big\| \cT \bb^{k} - \cT \bbstar \big\|^2
  \le   L \big\| \bb^k - \bbstar \big\|_\cQ^2
  \le L \nu^k \big\| \bb^0 - \bbstar \big\|_\cQ^2, 
\]
which concludes  the strong convergence of $x^k \rightarrow x^\star$, as $k\rightarrow \infty$.
\hfill 
\end{proof}

\vskip.2cm
The following results build the connection of the convergence properties with the cocoerciveness of $\cT$.
\begin{proposition}[Convergence of \eqref{banach}] \label{p_banach_coco}
Let $\bb^0\in D$, $\{\bb^k\}_{k \in \N}$ be a sequence generated by \eqref{banach}, with $\cT$ being $\cQ$--based $\beta$--cocoercive with $\beta \in [ 1, +\infty[$. Under Assumptions \ref{assume_1} and \ref{assume_2}-(i-ii), the following hold.

{\rm (i)} $\cT$ is $\cQ$--asymptotically regular.

{\rm (ii) [Sequential error]}   $\|\bb^{k+1 } -\bb^{k} \|_\cQ$ has the pointwise sublinear convergence rate of $\cO(1/\sqrt{k})$:
\[
\big\|\bb^{k +1} -\bb^{k} \big\|_\cQ
\le \frac{1}{\sqrt{k+1}} \cdot 
\sqrt{2\beta - 1} \big\|\bb^{0} -\bb^\star \big\|_\cQ,
\quad  \forall k \in \N.
\]

{\rm (iii) [$q$--linear convergence]} If $\beta \in \  ]1, +\infty [ $,  both $\|\bb^{k} -\bbstar \|_\cQ$ and $\|\bb^{k} -\bb^{k+1} \|_\cQ$  are $q$--linearly convergent with the rate of $\frac{1}{\sqrt{2\beta-1}}$.

{\rm (iv) [$r$--linear convergence]}  If $\beta \in [ 
\frac{ \sqrt{2}+1} {2}, +\infty [ $,  
 $\big\| \bb^k - \bb^{k+1} \big\|_\cQ$  is globally $r$--linearly convergent w.r.t. $\big\| \bb^0 - \bbstar \big\|_\cQ$:
\[
    \big\| \bb^k - \bb^{k+1} \big\|_\cQ
   \le  2\sqrt{ \beta- 1 }  \cdot  (2\beta-1)^{-\frac{k}{2} }
\big\| \bb^0 - \bbstar \big\|_\cQ.
\]
 
{\rm (v) [Weak/strong convergence in $\ran \cQ$]} If $\beta=1$ or  $\beta \in \ ]1, +\infty[$, there exists $x^\star \in \Fix \cT$, such that $\sqrt{\cQ} \bb^k \rightharpoonup \text{\rm or\ } \rightarrow \sqrt{\cQ}  x^\star$ respectively, as $k\rightarrow \infty$.

{\rm (vi) [Weak/strong convergence in $\cH$]}  Under Assumptions \ref{assume_1} and \ref{assume_2}, if $\beta=1$ or $\beta \in \ ]1, +\infty[$, then there exists $\bbstar \in \Fix \cT$, such that $\bb^k \rightharpoonup \text{\ or\ } \rightarrow \bbstar$ respectively, as $k\rightarrow \infty$.
\end{proposition}

\begin{proof}
By Lemma \ref{l_cocoercive}--(i), we have: $\cT \in \cF^\cQ_{\frac{1}{2\beta-1}, 1-\frac{1}{2\beta}}$. Taking $(\bb_1,x_2) = (\bb^k, \bbstar)$, or $(\bb_1,x_2)=( \bb^k,x^{k+1}) $  in \eqref{xi_alpha}, respectively,  we have:
\[
\left\{ \begin{array}{lll}
\big\|  \bb^{k+1} -\bbstar \big\|_\cQ^2
& \le &  \frac{1}{2\beta-1} \big\| \bb^k - \bbstar \big\|_\cQ^2
- \frac{1}{2\beta - 1} \big\|  \bb^k -  \bb^{k+1} \big\|_\cQ^2, \\
\big\|  \bb^{k+1} -\bb^{k+2} \big\|_\cQ^2
& \le &  \frac{1}{2\beta-1} \big\| \bb^k - \bb^{k+1} \big\|_\cQ^2.
\end{array} \right.
\]
The rest of proof is similar to  Proposition \ref{p_dist}, Theorem \ref{t_banach} and Corollary \ref{c_banach}. 
\hfill 
\end{proof}

\vskip.1cm
\begin{remark}
Theorem \ref{t_banach} and  Corollary \ref{c_banach} are closely linked to Proposition \ref{p_banach_coco}, if $\cT \in \cF^\cQ_{\xi,\alpha}$ is also $\beta$--cocoercive. This connection can be immediately obtained by Theorem \ref{t_lip}--(iii). Indeed, if $\xi \le \min\{\frac{1-\alpha}{\alpha}, 1\}$, 
$\cT \in \cF^\cQ_{\xi,\alpha}$ is $\cQ$--firmly nonexpansive (by Theorem \ref{t_lip}--(ii)), and also $\beta$--cocoercive with $\beta=\frac{1}{2} (1+ \frac{1}{1-\alpha+\alpha \xi^2}) \ge 1$ (by Theorem \ref{t_lip}--(iii)). According to Theorem \ref{t_banach} and  Corollary \ref{c_banach}, 
 $\xi \in\ ]0,1]$ is sufficient to guarantee the convergence, while $\cT$ is not necessarily $\cQ$--firmly nonexpansive. This implies that  the $\cQ$--firm nonexpansiveness of $\cT$ is an over--sufficient condition for the convergence of \eqref{banach}.

If $\cT \in \cF^\cQ_{\xi,\alpha}$ is $\beta$--cocoercive, 
$\xi \le \min\{\frac{1-\alpha}{\alpha}, 1\}$ guarantees the convergence (by Proposition \ref{p_banach_coco}), while $\cT$ is also $\cQ$--firmly nonexpansive (by Theorem \ref{t_lip}--(ii)). In this sense, Proposition \ref{p_banach_coco} is somewhat a special case of Theorem \ref{t_banach} and  Corollary \ref{c_banach}. Note that in Theorem \ref{t_banach} and  Corollary \ref{c_banach}, the convergence condition $\xi \in\ ] 0, 1]$ cannot guarantee the cocoerciveness of $\cT$. For instance, when $\alpha \in\  ] \frac{1}{2}, +\infty[ $ and $\xi \in\  ]\frac{1-\alpha}{\alpha}, 1]$, \eqref{banach} is convergent, but $\cT$ is not cocoercive.
\end{remark}

\subsection{Krasnosel'ski\u{\i}-Mann  algorithm}
\label{sec_km}
Consider the iteration:
\be \label{km}
\bb^{k+1} :=  \bb^k + \gamma(\cT \bb^k - \bb^k)
: = \cT_\gamma \bb^k,
\ee
where $\cT_\gamma = \cI - \gamma (\cI - \cT)$ and  $\cT \in \cF^\cQ_{\xi,\alpha}$.

\begin{corollary}[Convergence  of \eqref{km}] \label{c_km}
Let $\bb^0\in D$, $\{\bb^k\}_{k \in \N}$ be a sequence generated by \eqref{km}. Denote $\nu: = 1-\gamma\alpha + \gamma \alpha \xi^2$. Under Assumption \ref{assume_1} and \ref{assume_2}-(i-ii), if $\gamma \in \ ]0, 1/\alpha [$,  the following hold.

{\rm (i)} $\cT$ is $\cQ$--asymptotically regular.

{\rm (ii) [Sequential error]}   $\|\bb^{k+1 } -\bb^{k} \|_\cQ$ has the pointwise sublinear convergence rate of $\cO(1/\sqrt{k})$:
\[
\big\|\bb^{k +1} -\bb^{k} \big\|_\cQ
\le \frac{1}{\sqrt{k+1}}   \sqrt{ \frac{\gamma\alpha}{1- \gamma\alpha} }
\big\|\bb^{0} -\bb^\star \big\|_\cQ, \quad
\forall k \in \N.
\]

{\rm (iii) [$q$--linear convergence]} If $\xi \in \ ]0, 1[$,  both $\|\bb^{k} - x^\star \|_\cQ$ and $\|\bb^{k} -\bb^{k+1} \|_\cQ$  are $q$--linearly convergent with the rate of $\sqrt{\nu} $.

{\rm (iv) [$r$--linear convergence]}  If $\gamma \alpha \in \big] 1-\frac{1} {\sqrt{ 2} } , 1 \big[$,
$\xi \in \Big] 0, \sqrt{ 1 -  \frac{ 2- \sqrt{2}} { 2 \gamma \alpha} } \Big] $, 
 $\big\| \bb^k - \bb^{k+1} \big\|_\cQ$  is globally $r$--linearly convergent w.r.t. $\big\| \bb^0 - x^\star \big\|_\cQ$:
\[
    \big\| \bb^k - \bb^{k+1} \big\|_\cQ
   \le\sqrt{  \frac{2\gamma\alpha (1-\nu) } {(1-\gamma\alpha) \nu }  }
    \cdot \nu^{\frac{k+1}{2} }  
\big\| \bb^0 - \bbstar \big\|_\cQ. 
\]

{\rm (v) [Weak/strong convergence in $\ran \cQ$]} If $\xi=1$ or  $\xi \in \ ]0,1[$, there exists $x^\star \in \Fix \cT$, such that $\sqrt{\cQ} \bb^k \rightharpoonup \text{\rm or\ } \rightarrow \sqrt{\cQ}  x^\star$ respectively, as $k\rightarrow \infty$. 

{\rm (vi) [Weak/strong convergence in $\cH$]}  Under Assumptions \ref{assume_1} and \ref{assume_2}, if $\xi=1$ or $\xi \in \ ]0, 1 [$, then there exists $\bbstar \in \Fix \cT$, such that $\bb^k \rightharpoonup \text{\ or\ } \rightarrow \bbstar$ respectively, as $k\rightarrow \infty$.
\end{corollary}

\begin{proof} 
First, we claim that  $\Fix \cT_\gamma = \Fix \cT$. Indeed, $\bbstar \in \Fix \cT_\gamma \Longleftrightarrow \bbstar = \bbstar - \gamma (\bbstar - \cT \bbstar) \Longleftrightarrow \bbstar = \cT \bbstar \Longleftrightarrow \bbstar \in \Fix \cT$.

If $\gamma < \frac{1}{\alpha}$,  we deduce  by Theorem \ref{t_lip}--(iv) that: 
\[
\cT \in \cF^\cQ_{\xi,\alpha} \Longrightarrow \cR = \cI - \cT \in \cF^\cQ_{\frac{\alpha \xi}{1-\alpha}, 1- \alpha} 
\Longrightarrow \cT_\gamma = \cI - \gamma \cR  \in \cF^\cQ_{\xi, \gamma \alpha}.
\]

The rest of the proof is similar to Proposition \ref{p_dist}, Theorem \ref{t_banach} and Corollary \ref{c_banach},  just replacing $\alpha$ by $\gamma \alpha$, provided that $\gamma < \frac{1}{\alpha}$. 
\hfill
\end{proof}

\section{Application to metric resolvent}
\label{sec_appl}
\subsection{Basic properties}
Consider the metric resolvent\footnote{It is also called $F$-resolvent in \cite{hhb_resolvent} or warped resolvent \cite{warp}.}:
\be \label{t}
\cT:=(\cA+\cQ)^{-1} \cQ,
\ee
where $\cA: \cH \mapsto 2^\cH$ is a set-valued maximally monotone operator, $\cQ \in \cM_\cS^+$.  It is easy to show that  $\cT \in \cF^\cQ_{1, \frac{1}{2}}$, $\cI - \cT \in \cF^\cQ_{1, \frac{1}{2}}$ \cite{bredies_2017,fxue_gopt}. Furthermore,  if $\cA$ is $\mu$--strongly  monotone,    $\cT \in \cF^\cQ_{\frac{\|\cQ\|}{2\mu + \|\cQ\|}, \frac{2\mu +\|\cQ\|} {2\mu +2\|\cQ\|} }$,    $\cI - \cT  \in \cF^\cQ_{1, \frac{\|\cQ\|} { 2(\|\cQ\| +  \mu ) } }$. Then, the convergence properties of the Banach-Picard iteration:
\be  \label{ppa}
\bb^{k+1} := (\cA+\cQ)^{-1} \cQ \bb^k
\ee
immediately follow from Proposition \ref{p_dist}, Theorem \ref{t_banach} and Corollary \ref{c_banach}, by substituting $\xi$ and $\alpha$ with proper quantities, if the corresponding assumptions are satisfied. 

Considering the Krasnosel'ski\u{\i}-Mann  iteration:
\be  \label{rppa}
\bb^{k+1} := \bb^k + \gamma \big( (\cA+\cQ)^{-1} \cQ \bb^k
-\bb^k \big),
\ee
it is easy to show that $ \cT_\gamma \in \cF^\cQ_{1, \frac{\gamma}{2}}$ from the proof of Corollary \ref{c_km}. Furthermore,  if $\cA$ is $\mu$--strongly  monotone, $ \cT_\gamma \in  \cF^\cQ_{\frac{  \|\cQ\|}
{ 2  \mu + \|\cQ\|},    
\frac { \gamma( 2  \mu +  \|\cQ\|)} {2\mu+ 2 \|\cQ\|}  }$.  Then, the convergence properties of \eqref{rppa}  follow from Corollary \ref{c_km}, if the corresponding assumptions are fulfilled.

\subsection{Reinterpretation of primal-dual hybrid gradient   algorithm}
\label{sec_eg}
The primal-dual hybrid gradient (PDHG)  algorithm, for solving $\min_u\  f(u) + g(A u)$\footnote{Here, the functions $f$ and $g$ are assumed to be proper, lower semi-continuous and convex.}, is given as \cite{cp_2011,esser}: 
\be \label{pdhg}
\left\lfloor \begin{array}{lll}
 s^{k+1}   & := &  \prox_{\sigma g^*}  \big( s^k +\sigma 
 A u^k   \big) ,\\
u^{k+1}   & := &  \prox_{\tau f}  \big( u^k - \tau   
A^* (2s^{k+1} -  s^k )   \big).
\end{array} \right. 
\ee
It exactly fits into the form of  metric resolvent \eqref{ppa}: 
\be \label{ppa_pdhg}
\begin{bmatrix}   s^{k+1} \\  u^{k+1}
\end{bmatrix} = \bigg( \underbrace{  \begin{bmatrix}
\partial g^* & -A \\
A^*   & \partial f    \end{bmatrix} }_\cA
 + \underbrace{    \begin{bmatrix}
\frac{1}{ \sigma}  I  & A \\
A^*   & \frac{1}{ \tau}  I 
\end{bmatrix} }_\cQ  \bigg)^{-1} 
\underbrace{  \begin{bmatrix}
\frac{1}{ \sigma} I & A \\
 A^*   & \frac{1}{ \tau} I 
\end{bmatrix} }_\cQ   \begin{bmatrix}  
 s^k \\  u^{k}
\end{bmatrix}.
\ee

For this specific case of \eqref{ppa_pdhg}, we have the following basic observations:
\begin{itemize}
\item $\cA$ is maximally monotone;
\item $\cQ$ is self-adjoint and PSD, if $\tau\sigma \le \frac{1}{\|A^*A\|}$;
\item $\cT \in \cF^\cQ_{1, \frac{1}{2}}$;
\item $\Fix \cT = \zer \cA$.
\end{itemize}
Here, Assumption \ref{assume_2}-(ii) is satisfied, as long as $\zer \cA \ne \emptyset$, i.e. there exists a point $(u^\star,s^\star)$ satisfying  the Karush-Kuhn-Tucker conditions. This is a reasonable assumption under this context.

Based on the above results, it is needless to discuss the non-degenerate case when $\tau\sigma < \frac{1}{\|A^*A\|}$. We are mainly concerned with  the degenerate metric  when $\tau\sigma = \frac{1}{\|A^*A\|}$. We now claim that \eqref{ppa_pdhg} satisfies Assumption \ref{assume_2}-(iv). Indeed, 
\[
\big\|\cT x_1 - \cT x_2 \big\| \le 
\big\| (\cA+\cQ)^{-1} \cQ x_1 - (\cA+\cQ)^{-1} \cQ x_2 \big\| \le 
\big\| (\cA+\cQ)^{-1} \big\| \cdot \big\| \sqrt{ \cQ} \big\| \cdot \| x_1 - x_2 \big\|_\cQ ,
\]
where 
\[
 (\cA+\cQ)^{-1}: (s,u)\mapsto \big( \prox_{\sigma g^*}(\sigma s), \prox_{\tau f} (-2\tau A^* \prox_{\sigma g^*}(\sigma s) ) + \prox_{\tau f} (\tau u) \big).
\]
This is a composition of Lipschitz functions, and the Lipschitz constant $L$ is not relevant in this context. Finally, we have verified that \eqref{ppa_pdhg} with degenerate metric $\cQ$ satisfies all of Assumptions \ref{assume_1} and \ref{assume_2}, and thus, the results in Sect. \ref{sec_operator} and \ref{sec_fixed} can be applied.

\section{Concluding remarks}
We investigated in details the nonexpansive mappings in the context of arbitrary  metric, and particularly discussed the convergence  of the associated fixed-point  iterations under the setting of degenerate metric.

There are more prospective applications of our results. Besides from PDHG, more splitting algorithms can be reformulated as the metric resolvent, many of them correspond to degenerate metric. In addition, our results can be extended to analyze more related concepts, e.g. generalized proximity operator, Bregman proximal map, variable metric Fej\'{e}r sequence, especially equipped with degenerate metric.

\section{Data availability}
There is no associated data with this manuscript.

\section{Disclosure statement}
The author declares there are no conflicts of interest regarding the publication of this paper.

\bibliographystyle{siam}

\small{
\bibliography{refs}

\begin{thebibliography}{10}

\bibitem{attouch_2001}
{\sc F.~Alvarez and H.~Attouch}, {\em An inertial proximal method for maximal
  monotone operators via discretization of a nonlinear oscillator with
  damping}, Set-valued Analysis, 9 (2001), pp.~3--11.

\bibitem{attouch_2018}
{\sc H.~Attouch, J.~Peypouquet, and P.~Redont}, {\em Backward–forward
  algorithms for structured monotone inclusions in {H}ilbert spaces}, Journal
  of Mathematical Analysis and Applications, 457 (2018), pp.~1095--1117.

\bibitem{baillon}
{\sc J.~B. Baillon, R.~E. Bruck, and S.~Reich}, {\em On the asymptotic behavior
  of nonexpansive mappings and semigroups in {B}anach spaces}, Houston J.
  Math., 4 (1978), pp.~1--9.

\bibitem{bau_2003}
{\sc Heinz~H. Bauschke}, {\em The composition of projections onto closed convex
  sets in {H}ilbert space is asymptotically regular}, Proc. Amer. Math. Soc.,
  131 (2003), pp.~141--146.

\bibitem{plc_book}
{\sc Heinz~H. Bauschke and Patrick~L. Combettes}, {\em Convex Analysis and
  Monotone Operator Theory in Hilbert Spaces}, Second Edition, CMS Books in
  Mathematics, Springer, New York, NY, 2017.

\bibitem{hhb_corres}
{\sc Heinz~H. Bauschke, Sarah~M. Moffat, and Xianfu Wang}, {\em Firmly
  nonexpansive mappings and maximally monotone operators: Correspondence and
  duality}, Set-Valued and Variational Analysis, 20 (2012), pp.~131--153.

\bibitem{hhb_resolvent}
{\sc Heinz~H. Bauschke, Xianfu Wang, and L.~Yao}, {\em General resolvents for
  monotone operators: characterization and extension}, Biomedical Mathematics:
  Promising Directions in Imaging, Therapy Planning and Inverse Problems,
  (2010).

\bibitem{beck_book}
{\sc Amir Beck}, {\em First-Order Methods in Optimization}, SIAM-Society for
  Industrial and Applied Mathematics, 2017.

\bibitem{bot_2014}
{\sc R.I. Bo\c{t} and E.R. Csetnek}, {\em On the convergence rate of a
  forward-backward type primal-dual primal-dual splitting algorithm for convex
  optimization problems}, Optimization, 64 (2014), pp.~5--23.

\bibitem{res_17}
{\sc Jonathan~M. Borwein}, {\em Fifty years of maximal monotonicity},
  Optimization Letters, 4 (2010), pp.~473--490.

\bibitem{borwein}
{\sc Jon~M. Borwein and Brailey Sims}, {\em Nonexpansive mappings on {B}anach
  lattices and related topics}, Houston J. Math., 10 (1984), pp.~339--356.

\bibitem{bredies_2017}
{\sc K.~Bredies and H.P. Sun}, {\em A proximal point analysis of the
  preconditioned alternating direction method of multipliers}, J. Optim. Theory
  Appl., 173 (2017), pp.~878--907.

\bibitem{arias_infimal}
{\sc L.M. Brice\~{n}o Arias and F.~Rold\'{a}n}, {\em Resolvent of the parallel
  composition and proximity operator of the infimal postcomposition}, arXiv
  preprint: arXiv:2109.06771,  (2021).

\bibitem{warp}
{\sc Minh~N. B\`{u}i and Patrick~L. Combettes}, {\em Warped proximal iterations
  for monotone inclusions}, Journal of Mathematical Analysis and Applications,
  491 (2020), p.~124315.

\bibitem{cp_2011}
{\sc A.~Chambolle and T.~Pock}, {\em A first-order primal-dual algorithm for
  convex problems with applications to imaging}, J. Math. Imag. Vis., 40
  (2011), pp.~120--145.

\bibitem{entropic_tsp}
{\sc Hsiao-Han Chao and Lieven Vandenberghe}, {\em Entropic proximal operators
  for nonnegative trigonometric polynomials}, IEEE Transactions on Signal
  Processing, 66 (2018), pp.~4826--4838.

\bibitem{chaokan}
{\sc Kan Chao and Song Wen}, {\em The {M}oreau envelope function and proximal
  mapping in the sense of the {B}regman distance}, Nonlinear Analysis: Theory,
  Methods \& Applications, 75 (2012), pp.~1385--1399.

\bibitem{plc_fixed}
{\sc P.L. Combettes and J.C. Pesquet}, {\em Fixed point strategies in data
  science}, IEEE Transactions on Signal Processing, 69 (2021), pp.~3878--3905.

\bibitem{plc_vu}
{\sc P.L. Combettes and B.C. V\~{u}}, {\em Variable metric quasi-{F}ej\'{e}r
  monotonicity}, Nonlinear Analysis: Theory, Methods \& Applications, 78
  (2016), pp.~17--31.

\bibitem{plc}
{\sc P.L. Combettes and V.R. Wajs}, {\em Signal recovery by proximal
  forward-backward splitting}, Multiscale Modeling and Simulation, 4 (2005),
  pp.~1168--1200.

\bibitem{esser}
{\sc E.~Esser, X.~Zhang, and T.F. Chan}, {\em A general framework for a class
  of first order primal-dual algorithms for convex optimization in imaging
  science}, SIAM Journal on Imaging Sciences, 3 (2010), pp.~1015--1046.

\bibitem{boyd_control}
{\sc Pontus Giselsson and Stephen Boyd}, {\em Linear convergence and metric
  selection for {D}ouglas-{R}achford splitting and {ADMM}}, IEEE Transactions
  on Automatic Control, 62 (2017), pp.~532--544.

\bibitem{ppa_guler}
{\sc Osman G\"{u}ler}, {\em On the convergence of the proximal point algorithm
  for convex minimization}, SIAM J. Control Optim., 29 (1991), pp.~403--419.

\bibitem{hbs_prs}
{\sc B.~He, H.~Liu, Z.~Wang, and X.~Yuan}, {\em A strictly contractive
  {Peaceman}--{Rachford} splitting method for convex programming}, SIAM Journal
  on Optimization, 24 (2014), pp.~1011--1040.

\bibitem{hbs_jmiv_2017}
{\sc Bingsheng He, Feng Ma, and Xiaoming Yuan}, {\em An algorithmic framework
  of generalized primal-dual hybrid gradient methods for saddle point
  problems}, Journal of Mathematical Imaging and Vision, 58 (2017),
  pp.~279--293.

\bibitem{hbs_siam_2012}
{\sc Bingsheng He and Xiaoming Yuan}, {\em On the $\mathcal{O}(1/n)$
  convergence rate of the {D}ouglas-{R}achford alternating direction method},
  SIAM J. Numerical Analysis, 50 (2012), pp.~700--709.

\bibitem{hbs_yxm_2015}
{\sc B.~He and X.~Yuan}, {\em On non-ergodic convergence rate of
  {Douglas}--{Rachford} alternating direction method of multipliers},
  Numerische Mathematik, 130 (2015), pp.~567--577.

\bibitem{res_16}
{\sc Christian Kanzow and Yekini Shehu}, {\em Generalized
  {K}rasnosel'ski\u{\i}--{M}ann--type iterations for nonexpansive mappings in
  {H}ilbert spaces}, Computational Optimization and Applications, 67 (2017),
  pp.~595--620.

\bibitem{kirk}
{\sc W.~A. Kirk}, {\em Nonexpansive mappings and asymptotic regularity},
  Nonlinear Analysis, 40 (2000), pp.~323--332.

\bibitem{koh}
{\sc Ulrich Kohlenbach}, {\em A polynomial rate of asymptotic regularity for
  compositions of projections in {H}ilbert space}, Foundations of Computational
  Mathematics, 19 (2019), pp.~83--99.

\bibitem{prox_acha}
{\sc Qia Li and Na~Zhang}, {\em Fast proximity-gradient algorithms for
  structured convex optimization problems}, Applied and Computational Harmonic
  Analysis, 41 (2016), pp.~491--517.

\bibitem{ljw_mapr}
{\sc Jingwei Liang, Jalal Fadili, and Gabriel Peyr\'{e}}, {\em Convergence
  rates with inexact non-expansive operators}, Mathematical Programming, 159
  (2016), pp.~403--434.

\bibitem{bregman_vietnam}
{\sc Quang~Van Nguyen}, {\em Forward-backward splitting with {B}regman
  distances}, Vietnam Journal of Mathematics, 45 (2017), pp.~519--539.

\bibitem{opial}
{\sc Z.~Opial}, {\em Weak convergence of the sequence of successive
  approximations for nonexpansive mappings}, Bull. Amer. Math. Soc., 73 (1967),
  pp.~591--597.

\bibitem{rtr_book}
{\sc R.~T. Rockafellar}, {\em Convex analysis}, Princeton Landmarks in
  Mathematics and Physics, Princeton University Press, 1996.

\bibitem{rtr_book_2}
{\sc R.~Tyrrell Rockafellar and Roger J-B Wets}, {\em Variational Analysis},
  Springer, Grundlehren der Mathematischen Wissenschaft, vol. 317, 2004.

\bibitem{res_14}
{\sc Tomonari Suzuki}, {\em Fixed point theorems and convergence theorems for
  some generalized nonexpansive mappings}, J. Math. Anal. Appl., 340 (2008),
  pp.~1088--1095.

\bibitem{res_3}
{\sc M.~Teboulle}, {\em Entropic proximal mappings with applications to
  nonlinear programming}, Math. Oper. Res., 17 (1992), pp.~670--690.

\bibitem{teboulle_2018}
\leavevmode\vrule height 2pt depth -1.6pt width 23pt, {\em A simplified view of
  first order methods for optimization}, Math. Program., Ser. B, 170 (2018),
  pp.~67--96.

\bibitem{vu_2013}
{\sc B.C. V\~{u}}, {\em A splitting algorithm for dual monotone inclusions
  involving cocoercive operators}, Adv. Comput. Math., 38 (2013), pp.~667--681.

\bibitem{fxue_gopt}
{\sc Feng Xue}, {\em Some extensions of the operator splitting schemes based on
  {L}agrangian and primal-dual: A unified proximal point analysis},
  Optimization, {\rm DOI: 10.1080/02331934.2022.2057309},  (2022).

\end{thebibliography}
}

\end{document}